%% file: tcvarieties.tex
\documentclass{amsart}
\usepackage{amsmath,amssymb,amsthm,amsfonts,amscd,mathrsfs,mathtools}
\usepackage[utf8]{inputenc}
\usepackage{graphicx}
\usepackage{enumitem}
\usepackage[hyphens]{url}
\usepackage{caption}
\usepackage{hyperref}
\usepackage{tikz-cd}

\theoremstyle{plain}
    \newtheorem{thm}{Theorem}[section]
    \newtheorem{lem}[thm]   {Lemma}
    
    \newtheorem{prop}[thm]  {Proposition}

\theoremstyle{definition}
    \newtheorem{defn}[thm]  {Definition}

\input{macros.tex}

\begin{document}

\title[An upper bound on TC of discriminantal varieties]{An upper bound on the topological complexity of discriminantal varieties}

\author{Andrea Bianchi}

\address{Mathematics Institute,
University of Copenhagen,
Universitetsparken 5, Copenhagen,
Denmark
}

\email{anbi@math.ku.dk}

\date{\today}


\begin{abstract}
We give an upper bound on the topological complexity of varieties $\V$ obtained as complements
in $\C^m$ of the zero locus of a polynomial. As an application, we determine the topological complexity
of unordered configuration spaces of the plane.
\end{abstract}

%\thanks{}

\maketitle
\section{Introduction}
Let $X$ be a path-connected topological space. 
% The motion planning problem on $X$ asks
% to assign to each couple $x_1,x_2$ of points in $X$
% a continuous path $\gamma\colon[0,1]\to X$ joining $x_1$ to $x_2$. This problem is central
% to robotics (think of $X$ as the space of configurations of a robot, or a system of robots);
% a natural question is \emph{how robust a motion planner on $X$ can be}, or equivalently,
% \emph{how many discontinuities do we have to allow in our choice of paths}.
The topological complexity $\TC(X)$ is a numerical invariant introduced by Farber \cite{Farber} to
measure the complexity of a motion planner on $X$.
Denote by $X^{[0,1]}$ the path-space of $X$, whose points are all continuous maps $\gamma\colon[0,1]\to X$,
and by $p_X\colon X^{[0,1]}\to X\times X$ the fibration which assigns to each path $\gamma\colon[0,1]\to X$
(i.e. each element $\gamma\in X^{[0,1]}$) the couple $(\gamma(0),\gamma(1))\in X\times X$.
\begin{defn}[Topological complexity]
\label{defn:TC}
A \emph{motion planner} on a path-connected space $X$ is the datum of:
\begin{itemize}
 \item a covering of $X\times X$ with open sets $U_0,\dots,U_k$;
 \item for each $U_i\subset X\times X$, a section $s_i\colon U_i\to X\times X$ of the fibration $p_X$.
\end{itemize}
The \emph{topological complexity} of $X$, denoted $\TC(X)$, is the minimum $k$ for which there exists a motion planner
on $X$ with $k+1$ open sets $U_0,\dots,U_k$ as above.
\end{defn}
The connection of the previous definition to robotics is immediate: if $X$ is the space of configurations
that a robot (or a system of robots) can assume, a motion planner on $X$ will give, for
each couple $x_1,x_2$ of points in $X$, a path to pass from the configuration $x_1$ to the configuration $x_2$.
The number $k$ measures the ``robustness'' of a motion planner: the smaller is $k$, the less
the motion planner is sensitive to small perturbations of $x_1$ and $x_2$ and thus the more it is reliable.

A variation of the notion of topological complexity is the \emph{equivariant topological complexity},
introduced by Colman and Grant \cite{CG12}.
Let $X$ be a path-connected space with an continuous action of a topological group $G$;
then $G$ acts diagonally on $X\times X$ and on $X^{[0,1]}$ (where in the second case the ``diagonal''
is the natural map $G\mapsto G^{[0,1]}$ sending $g\in G$ to the constant function $[0,1]\to G$ with
value $g$). Note that the fibration $p_X\colon X^{[0,1]}\to X\times X$ is a $G$-equivariant map.
\begin{defn}[Equivariant topological complexity]
A $G$-motion planner on $X$ is given by:
\begin{itemize}
 \item a covering of $X\times X$ by $G$-invariant open sets $U_0,\dots,U_k$;
 \item for each $U_i\subset X\times X$, a $G$-equivariant section $s_i\colon U_i\to X\times X$ of the fibration $p_X$.
\end{itemize}
The \emph{$G$-equivariant topological complexity} of $X$, denoted $\TC_G(X)$, is the minimum $k$ for which there exists a $G$-motion planner on $X$ with $k+1$ open sets $U_0,\dots,U_k$ as above.
\end{defn}
Clearly a $G$-motion planner is also a motion planner, hence $\TC(X)\leq \TC_G(X)$ for all path-connected $G$-spaces $X$.

In this paper we address the problem of bounding from above the topological complexity of 
algebraic varieties $\V$
obtained as complements in $\C^k$ of the zero locus of a polynomial.

\begin{defn}
 \label{defn:discriminantal}
 A \emph{discriminantal variety} is a complex algebraic variety $\V$ obtained as
 the complement in $\C^m$ of the zero locus of a non-zero polynomial $\Delta(z_1,\dots,z_m)$,
 called the \emph{discriminant}.
\end{defn}
As we will see, discriminantal varieties
often carry a natural, nice action of an $s$-dimensional torus $\T^s=(\Sone)^s$, so that we can compare $\TC(\V)$
and $\TC_{\T^s}(\V)$.

The motivating
example for considering discriminantal varieties is given by
\emph{configuration spaces of points in the plane}, which come usually in two different flavours.
\begin{defn}
 \label{defn:conf}
 The \emph{ordered configuration space} of $n$ points in $\C$ is the space
 \[
  F_n=\set{(w_1,\dots,w_n)\in\C^n\ | w_i\neq w_j\ \forall i\neq j}.
 \]
 There is a natural action of the symmetric group $\fS_n$ on $F_n$, by permuting the labels of a configuration.
 The quotient space is the \emph{unordered configuration space} of $n$ points in $\C$, and is denoted
 \[
  C_n:=F_n/\fS_n.
 \]
\end{defn}
Roughly speaking, the topological complexities $\TC(F_n)$ and $\TC(C_n)$ measure
the complexity of a motion planner which coordinates $n$ robots in the plane,
avoiding collisions between them. The difference between the two cases is subtle:
in the first case each robot is labeled by a number $1,\dots,n$, and the motion
planner takes as input, for each robot, its initial and its final position;
in the second case the robots are indistinguishable, and the motion planner's
task is to move all robots together, starting from the set of $n$ initial positions, so that
at the end the robots occupy (in some order) the $n$ final positions.

One can consider the action of $\T=\Sone=SO(2)$ by rotations on both $F_n$ and $C_n$:
then the topological complexities $\TC_{\T}(F_n)$ and $\TC_{\T}(C_n)$ measure
the complexity of a motion planner coordinating $n$ ordered or unordered 
robots in the plane, with the additional hypothesis that the motion planner should
be equivariant with respect to a ``change of coordinate system'' by rotation.
A motion planner with this property might be particularly useful, for instance,
because it prescribes to the robots the same behaviour independently of the spatial
point of view of the observer who measures their positions while coordinating them.

The topological complexity of $F_n$ was computed by Farber and Yuzvinsky \cite{FY} and is equal to
$2n-3$. For the space $C_n$, Recio-Mitter and the author \cite{BianchiRecio} proved
the bounds $2n-2\lfloor\sqrt{n/2}\rfloor-3\leq \TC(C_n)\leq 2n-2$, conjectured
that $\TC(C_n)=2n-3$, and showed that the lower bound $2n-3\leq \TC(C_n)$ would follow 
from a proof that the cohomological dimension of $[P_n,P_n]$
is equal to $n-2$, where $P_n=\pi_1(F_n)$ is the \emph{pure braid group on $n$ strands}
and $[P_n,P_n]$ is its commutator subgroup. The equality $[P_n,P_n]=n-2$ was then
proved by the author \cite{Bianchi:HcPn}. In this paper we prove the upper bound for $\TC(C_n)$,
and compute also the $\T$-equivariant topological complexities of $F_n$ and $C_n$.
\begin{thm}
 \label{thm:upperbound}
 Consider the action of $\T=\Sone=SO(2)$ on $F_n$ and $C_n$ by rotation of configurations of points
 in the plane. For all $n\geq 2$ we have $\TC_{\T}(F_n)\leq 2n-3$ and $\TC_{\T}(C_n)\leq 2n-3$. 
\end{thm}
Note that, together with the already known lower bound extimate, we obtain the chain of equalities
$\TC(F_n)=\TC_{\T}(F_n)=\TC(C_n)=\TC_{\T}(C_n)=2n-3$ for all $n\geq 2$.
We will obtain Theorem \ref{thm:upperbound} as an application of the following, more general statement.
\begin{thm}
\label{thm:main}
 Let $\V\subset\C^m$ be a discriminantal variety. Assume
 that, for some $s\geq 0$, the torus $\T^s=(\Sone)^s$ acts \emph{by scalar multiplication} on $\V$ (see Definition
 \ref{defn:scalaraction}), and assume that all stabilisers of the action have dimension $\leq t$. Then we have
 \[
  \TC(\V)\leq \TC_{\T^s}(\V)\leq 2m-s+t.
 \]
\end{thm}

The paper is organised as follows. In Sections 2, 3 and 4 we recall some standard material about actions
of groups on spaces, equivariant Morse theory and CW complexes in the equivariant setting. The exposition
is in a few points non-standard, and it is adapted to the content of the rest of the paper.
In Section 5 we prove Theorem \ref{thm:main} up to finding a $\T^s$-invariant Morse function
on $\V\times\V$ which satisfies some additional properties. The existence of such a function
is shown in Section 6. Finally, in Section 7 we apply Theorem \ref{thm:main} to prove
Theorem \ref{thm:upperbound}, and in Section 8 we briefly discuss the problem of finding an explicit, $\T$-equivariant
motion planner on $C_n$ (the case of $F_n$ being quite analogue).

Theorem \ref{thm:upperbound} was first announced at the UMI-SIMAI-PTM joint meeting in September 2018,
where I received some useful comments by Mark Grant and Paolo Salvatore, whom I thank.
I would also like to thank David Recio-Mitter for many useful conversations about motion planners,
and Mirko Mauri for pointing me to the Lefschetz hyperplane theorem.

% \section{Topological complexity and Euclidean neighbourhood retracts}
% If $X$ is an ENR (Euclidean neighbourhood retract), e.g. $X$ is homeomorphic to a CW complex,
% then instead of \emph{covering} $X\times X$ with open sets, one can \emph{decompose} $X\times X$ into disjoint ENRs $E_0,\dots,E_k$.
% 
% \begin{lem}[\cite{Farber06}]
% \label{lem:ENR}
% Recall Definition \ref{defn:TC}.
% If $X$ is an ENR, $\TC(X)$ is equal to the smallest $k\geq0$
% such that there exists a decomposition $X\times X=E_0\sqcup E_1\sqcup\cdots\sqcup E_k$
% into $k+1$ disjoint ENRs, on each of which the fibration $p_X$ admits a section
% $s_i\colon E_i\to X^{[0,1]}$.
% The existence of such a section $s_i\colon E_i\to X^I$ is equivalent to the existence of a deformation of $E_i$ into the diagonal of $X\times X$, i.e. a homotopy between the inclusion $E_i\hookrightarrow X\times X$ and a map whose image lies entirely in the diagonal.
% \end{lem}
% In the rest of the article we will work with this equivalent definition.

\section{Actions of groups on spaces}
Before starting we introduce some classical notions and some notation related to spaces
with an action of a group. We fix a compact, connected Lie group $G$ throughout the next three sections.

\begin{defn}
 \label{defn:Gequivhequiv}
 Two $G$-spaces $X$ and $Y$ are $G$-equivariantly homotopy equivalent if there are two $G$-equivariant
 maps $\alpha\colon X\to Y$ and $\beta\colon Y\to X$ such that $\beta\circ\alpha\colon X\to X$
 is homotopic to $\Id_X$ through $G$-equivariant maps $X\to X$, and $\alpha\circ\beta$ is
 homotopic to $\Id_Y$ through $G$-equivariant maps $Y\to Y$.
\end{defn}

% We will later also need a relative version of the previous definition.
% \begin{defn}
%  \label{defn:relGequivhequiv}
%  Let $(X,A)$ and $(Y,B)$ be couples of $G$-spaces, i.e. $A\subset X$ and $B\subset Y$ are
%  inclusions of $G$-spaces. We say that $(X,A)$ and $(Y,B)$ are $G$-equivariantly homotopy equivalent if there are two $G$-equivariant
%  maps of couples $\alpha\colon (X,A)\to (Y,B)$ and $\beta\colon (Y,B)\to (X,A)$ such that $\beta\circ\alpha\colon X\to X$
%  is homotopic to $\Id_X$ through $G$-equivariant maps $(X,A)\to (X,A)$, and $\alpha\circ\beta$ is
%  homotopic to $\Id_Y$ through $G$-equivariant maps $(Y,B)\to (Y,B)$.
% \end{defn}
\begin{defn}
 \label{defn:XH}
 Let $H<G$ be a closed subgroup and let $X$ be a $G$-space. We denote by $X^H\subset X$ the subspace
 of points $x\in X$ which are fixed by the action of $H$, i.e. $H\leq \mathrm{Stab}_x\leq G$.
\end{defn}
Note that if $X$ and $Y$ are $G$-homotopy equivalent $G$-spaces, then $X^H$ and $Y^H$ are homotopy
equivalent spaces for all closed $H<G$.
The following definition slightly differs from the one usually given in the literature, e.g. in \cite{CG12}.
\begin{defn}
\label{defn:Gconnected}
A $G$-space $X$ is $G$-connected if $X^H$ is path-connected \emph{or empty} for all closed subgroups $H<G$.
\end{defn}
The $G$-equivariant Lusternik-Schnirelman category was first introduced in \cite{Marzantowicz}.
We give a slightly non-standard definition of it.
\begin{defn}
\label{defn:GLS}
 Let $X$ be a $G$-space; an open set $U\subset X$ is \emph{$G$-categorical} if $U$ is $G$-invariant
 and there is a $G$-equivariant homotopy $U\times[0,1]\to X$ starting from the inclusion $U\hookrightarrow X$
 and ending with a map with values in a \emph{finite union
 of $G$-orbits} of $X$. The $G$-equivariant
 Lusternik-Schnirelman category of $X$, denoted $\cat_G(X)$, is the minimum $k$ such that $X$ can be covered by $k+1$
 $G$-categorical open sets $U_0,\dots,U_k$.
\end{defn}
The previous definition differs from the usual one because we allow a $G$-categorical set $U\subset X$
to deform $G$-equivariantly into the union of more than a single $G$-orbit of $X$ (clearly, this is in principle possible
only if $U$ is disconnected).

Similarly as the classical Lusternik-Schnirelman category, $\cat_G$
is a $G$-homotopy invariant of $G$-spaces: if two $G$-spaces $X$ and $Y$ are $G$-equivariantly
homotopy equivalent, then $\cat_G(X)=\cat_G(Y)$. The argument of the proof is the usual one also with our
slightly different definition, and we repete it for completeness.
Let $\alpha\colon X\to Y$ be a $G$-equivariant homotopy equivalence with $G$-homotopy inverse $\beta\colon Y\to X$,
and suppose that $Y$ is covered by $G$-categorical sets $U_0,\dots,U_k$ (in the sense of Definition \ref{defn:GLS});
then we claim that $V_0=\alpha^{-1}(U_0),\dots,V_k=\alpha^{-1}(U_k)$ are $G$-categorical sets covering $X$. To see that
$V_i$ is $G$-categorical, note that the $G$-homotopy $\beta\circ\alpha\simeq_G\Id_X$
restricts to a $G$-homotopy from the inclusion $V_i\hookrightarrow X$ to the composition
$\beta\circ\alpha\colon V_i\to X$; moreover $\alpha\colon V_i\to Y$ has image in $U_i$, so we can $G$-equivariantly
deform $\alpha\colon V_i\to Y$ to a map $V_i\to Y$ with image contained in finitely many orbits. Composing
this latter homotopy with $\beta$ yields a $G$-homotopy from $\beta\circ\alpha\colon V_i\to X$
to a map $V_i\to X$ with image in finitely many orbits.

The following is \cite[Proposition 5.6]{CG12}.
\begin{prop}
 \label{prop:TCvscat}
 Let $X$ be a $G$-connected $G$-space; then $\TC_G(X)\leq\cat_G(X\times X)$, where $X\times X$ is endowed
 with the diagonal action of $G$.
\end{prop}
The proof given in \cite{CG12} works also using our slightly different definitions. The main difference
is that, in our setting, we have to remark that
for each closed subgroup $H<G$, either $(X\times X)^H=X^H\times X^H$ is empty, or it is path-connected
and intersects the diagonal of $X\times X$ in some point.

\section{Equivariant Morse theory}
Recall that if $M$ is a smooth manifold and $f\colon M\to\R$ is a proper and bounded below
Morse function, then $M$ is homotopy equivalent
to a CW complex $X$ whose $k$-cells are in one-to-one correspondence with
critical points of $f$ of index $k$, for all $k\geq 0$.

In this section we establish the analogue result for manifolds $M$ endowed with a smooth
action of a compact, connected Lie group $G$: if $M$ is endowed with a $G$-invariant Morse function,
then $M$ is $G$-equivariantly
homotopy equivalent to a $G$-cell complex $X$ whose $k$-cells are in one-to-one correspondence
with critical manifolds $N\subset M$ of index $k$.

We will consider equivariant Morse theory and equivariant cell complexes
in the general case, although later we will specialise
to manifolds $M$ with an action of a torus $\T^s$. The material of this
section is an adaptation of the results of \cite[§4]{Wasserman}.

Let $M$ be a smooth real manifold,
% such that there is a smooth compact
% manifold with boundary $\bar M$ whose interior is $M$.
let $G$ be a compact Lie group,
and assume $G$ acts smoothly (on left) on $M$.
% the action of $g\in G$ on $M$ is denoted $p\in M\mapsto g\cdot p\in M$.
% For simplicity we assume that the action of $G$ on $M$ extends to an action on $\bar M$.

% We fix a $G$-equivariant Riemannian metric on $M$: this can be done by
% first fixing a Riemannian metric on $M$ and then averaging this metric by
% the action of $G$. We do not insist that the Riemannian metric is complete, and in
% our applications it will be convenient to choose a non-complete metric.

Let $f\colon M\to\R$ be a smooth function, and assume that $f$ is \emph{proper} (i.e.
for all closed interval $[a,b]\subset\R$ the preimage $f^{-1}([a,b])\subset M$ is compact)
and $f$ is $G$-invariant.
% (i.e. for all $p\in M$ and $g\in G$ we have $f(p)=f(g\cdot p)$).

\begin{defn}
 \label{defn:criticalmanifold}
A submanifold $N\subset M$ is called \emph{critical} for $f$ if the following conditions hold:
\begin{itemize}
 \item $N$ is an orbit of the action of $G$ on $M$;
 \item every point $p\in N$ is critical for $f$, i.e. $df(p)=0$.
\end{itemize}
\end{defn}
If $N$ is critical for $f$ and $p\in N$, then the Hessian $H(f)_p$ is a well-defined
bilinear form on $T_p(M)$, with $T_p(N)\subset\ker H(f)_p$, i.e. $H(f)\left<v,w\right>=0$
for all $v\in T_p(N)$ and $w\in T_p(M)$. It follows that there is an induced bilinear
form $\bar H(f)_p$ on the vector space $T_p(M)/T_p(N)$.
\begin{defn}
 \label{defn:nondegeneratecriticalmanifold}
 A critical manifold $N\subset M$ for $f$ is \emph{non-degenerate} if $\bar H(f)_p$
 is a non-degenerate bilinear form on $T_p(M)/T_p(N)$. The \emph{index}
 of $N$, denoted $\ind(N)$, is the negativity index $\iota_-$ of the bilinear form
 $\bar H(f)_p$ on $T_p(M)/T_p(N)$.
\end{defn}
To see that the index of $N$ does not depend on the choice of $p\in N$,
note that the signature of $\bar H(f)_p$ has always the form $(\iota_0(p),\iota_-(p),\iota_+(p))$
with $\iota_0(p)=0$, for all $p\in N$; since the positivity and negativity indices $\iota_+$ and $\iota_-$
are lower semicontinuous functions $N\to\mathbb{N}$ and the sum $\iota_-+\iota_+$ is constantly
equal to $\dim M-\dim N$ on $N$, it follows that $\iota_+$ and $\iota_-$ are continuous,
hence constant functions on $N$. Thus Definition \ref{defn:nondegeneratecriticalmanifold} is well-posed.

\begin{defn}
 \label{defn:GMorse}
 A $G$-invariant function $f\colon M\to \R$ is a $G$-invariant Morse function if
 it is bounded below, is proper and the critical locus of $f$ is the union of finitely
 many disjoint submanifolds $N_1,\dots,N_r\subset M$.
\end{defn}
Lemma 4.8 in \cite{Wasserman} is stated only for closed manifolds, but it
can be generalised to the following Lemma with the same proof.
\begin{lem}
 \label{lem:density}
 Let $\tilde f\colon M\to \R$ be a smooth, proper, $G$-invariant function, and assume that $\tilde f$ has no critical points outside a compact set $K\subset M$. Let $U$ be a $G$-invariant, relatively compact
 open neighbourhood of $K$ in $M$. Then for all $l\geq 1$ there is a $G$-invariant Morse
 function $f\colon M\to \R$
 such that $\tilde f-f$ vanishes outside $U$, and $\tilde f$ is arbitrarily close to $f$
 in the $C^l$ norm. 
\end{lem}
 We will need the case $l=2$. To state the main result about $G$-invariant Morse functions we need
 the following definition.
 
\begin{defn}
 \label{defn:Gcellcomplex}
 Let $H<G$ be a closed subgroup, let $R$ be an orthogonal representation of $H$ of real dimension $k$,
 and denote by $D(R)\subset R$ the closed unit ball, and by $S(R)=\partial D(R)\subset D(R)$ the unit sphere. We say that the $G$-space $Y$ is obtained from the $G$-space $X$ by a cell attachment of type $(H,R)$
 if there is a pushout diagram of $G$-spaces
 \[
  \begin{tikzcd}
   G\times_H S(R) \ar[r,"\phi"]\ar[d,"\subset"] & X \ar[d] \\
   G\times_H D(R) \ar[r] & Y
  \end{tikzcd}
 \]
 for some $G$-equivariant map $\phi\colon G\times_H S(R)\to X$. We denote by $\psi\colon G\times_H D(R)\to Y$ the induced map, which is called a \emph{characteristic map} of a cell, whereas $\phi$ is referred
 to as an \emph{attaching map}. We represent a cell by its characteristic map $\psi$; the \emph{index}
 of $\psi$ is defined as the real dimension of $R$.
 
 A \emph{finite} $G$-cell complex $X$ is a space obtained starting from $\emptyset$ and applying finitely
 many cell attachments as above. The \emph{index} of a $G$-cell complex $X$, denoted by $\ind(X)$, is the maximal index of its cells.
\end{defn}
Clearly the first cell that we attach to $\emptyset$ while constructing a $G$-cell complex $X$ must have
index $0$. Note however that we do not require in Definition \ref{defn:Gcellcomplex} that the cells are attached
in an order for which the index is weakly increasing; in particular, if $X$ is a finite $G$-cell complex
and $H$, $R$ and $\phi$ are as in Definition \ref{defn:Gcellcomplex}, it may happen that
$\phi$ hits points belonging to cells with index greater or equal than the one we are currently attaching.
It is not true in general that we can homotope $\phi$ in a $G$-equivariant way to a new attaching
map $\phi'$ hitting only points of $X$ belonging to cells of index lower than $\dim R$: in other words,
there is no $G$-equivariant cellular approximation theorem for $G$-cell complexes as in Definition \ref{defn:Gcellcomplex}.

The index $\ind(X)$ coincides with the geometric dimension of $X$ (seen as a space homeomorphic to
a plain CW complex) only when all $G$-cells of $X$ are defined using subgroups $H<G$ of finite index,
i.e. codimension 0. Our case of interest will be quite the opposite, namely the one in which every
$G$-cell is defined using a subgroup $H<G$ of small dimension.

Corollary 4.11 in \cite{Wasserman} is stated for compact manifolds, but it
can be generalised, with the same proof, to the following theorem.
\begin{thm}
 \label{thm:Wasserman}
 Let $M$ be a smooth manifold with a smooth action of $G$, and let $f\colon M\to \R$ be a $G$-invariant
 Morse function. Then $M$ is $G$-equivariantly homotopy equivalent to a finite $G$-cell complex
 $X$, whose cells $\psi_1,\dots,\psi_r$ are in one-to-one correspondence with the critical manifolds
 $N_1,\dots,N_r$ of $f$, such that $\ind(\psi_i)=\ind(N_i)$ for all $1\leq i\leq r$.
\end{thm}

\section{Equivariant CW complexes}
% We fix a compact, connected Lie group $G$ also throughout this section.
The results of this section are probably standard, but I could not find a source stating
them explicitly; nevertheless they are straightforward consequences of results that can be
found in \cite[II]{tomDieck} and \cite{Illman}, which are the main references for this section.

Our aim is to replace, up to $G$-equivariant homotopy equivalence, a $G$-cell complex by a $G$-CW
complex, which is defined as follows.
\begin{defn}
 \label{defn:GCW}
A finite $G$-pre-CW complex $X$ is a finite $G$-cell complex $X$ obtained by attaching cells of type $(H,\epsilon^{k})$, where $\epsilon^k$ is the trivial $k$-dimensional representation of some closed subgroup $H<G$, for some $k\geq 0$.

A finite $G$-CW complex $X$ is a particular type of finite $G$-pre-CW-complex: it is a $G$-space
$X$ filtered by $G$-invariant subspaces $\emptyset=X_{-1}\subset X_0\subset X_1\subset\dots\subset X_\nu=X$, for some $\nu\geq 0$,
such that, for all $0\leq k\leq \nu$, $X_k$ is obtained from $X_{k-1}$ by \emph{simultaneously} attaching
finitely many $G$-cells of type $\epsilon^k$. More precisely, for all $0\leq k\leq \nu$ there are:
\begin{itemize}
 \item a finite set $I_k$;
 \item a closed subgroup $H_j<G$ for all $j\in I_k$;
%  \item an $i$-dimensional representation $R_j$ of $H_j$ for all $j\in I_i$;
 \item a $G$-equivariant map $\phi_j\colon G\times_{H_j} S(\epsilon^k)\to X_{k-1}$ for all $j\in I_k$,
\end{itemize}
such that $X_k$ is obtained as pushout of the following diagram
\[
 \begin{tikzcd}
  \coprod_{j\in I_k} G\times_{H_j}S(\epsilon^k) \ar[d,"\subset"]\ar[r,"\coprod \phi_j"] & X_{k-1}\ar[d]\\
  \coprod_{j\in I_k} G\times_{H_j}D(\epsilon^k) \ar[r] & X_k.
 \end{tikzcd}
\]
\end{defn}

The following proposition is a direct consequence of \cite[Theorem II.2.1]{tomDieck}.
\begin{prop}
 \label{prop:cellularapproximation}
 Every $G$-pre-CW complex $X$ is $G$-equivariantly homotopy equivalent to a $G$-CW complex
 $Y$ with the same number of cells in every index.
\end{prop}

Thus our aim reduces to replacing a $G$-cell complex by a $G$-pre-CW complex up to $G$-equivariant
homotopy equivalence. Before proceeding, we state the following theorem, which is a weak form of
\cite[Theorem 7.1]{Illman}.
\begin{thm}
 \label{thm:decomposition}
 Let $W$ be a smooth, compact manifold of dimension $k$ with a smooth action of $G$; then $W$ is homeomorphic to a $G$-CW complex with $G$-cells of index $\leq k$.
\end{thm}

Suppose now that we are given a $G$-pre-CW complex $X$, a closed subgroup $H<G$, a possibly non-trivial
representation $R$ of $H$ of some dimension $k$ and a $G$-equivariant map $\phi\colon G\times_H S(R)\to X$.
Then $G\times_H S(R)$ is a compact manifold
(it is the total space of a bundle over $G/H$ with fibre $S(R)$), so by Theorem \ref{thm:decomposition}
there is a decomposition of $G\times_H S(R)$ as a $G$-CW complex.

We can extend this to a structure of $G$-CW-complex 
on $G\times_H D(R)$, by using a cell of index $0$ to represent
the zero section of the disc bundle $G\times_H D(R)\to G/H$,
% (clearly this cell is constructed
% using $H$ as subgroup of $G$)
and then by coning off the $G$-CW structure on $G\times_H S(R)$
inside $G\times_H D(R)$: note that the cells filling the cone have index precisely 1 more than
the corresponding ones in $G\times_H S(R)$.
We thus obtain a structure of $G$-pre-CW complex on
the pushout of $G\times_H D(R)$ and $X$ along $\phi$.

Repeating this argument for all cells in a $G$-cell complex, we obtain a decomposition
of any $G$-cell complex $X$ as a $G$-pre-CW complex: the total number of cells will increase,
but the index of $X$ is the same when $X$ is considered as a $G$-cell complex (before the decomposition)
or as a $G$-pre-CW complex (after the decomposition).

Putting together
this argument, Theorem \ref{thm:Wasserman} and Proposition \ref{prop:cellularapproximation} we obtain
the following theorem, for which I could not find a direct reference in the literature.
\begin{thm}
 \label{thm:GMorsefinal}
 Let $M$ be a smooth manifold with a smooth action of $G$, and let $f\colon M\to\R$ be a $G$-invariant Morse function whose critical manifolds have index $\leq k$.
 Then $M$ is $G$-equivariantly homotopy equivalent to a finite $G$-CW complex $X$
 with $\ind(X)\leq k$.
\end{thm}

The index of a $G$-CW complex $X$ is an upperbound for the $G$-equivariant Lusternik-Schnirelman category
$\cat_G(X)$, as we see in the following proposition.
\begin{prop}
 \label{prop:CWLS}
 Let $X$ be a finite $G$-CW complex of index $\nu$; then $\cat_G(X)\leq \nu$.
\end{prop}
\begin{proof}
 Let $\emptyset=X_{-1}\subset X_0\subset\dots\subset X_\nu$ be the filtration of $X$ as in Definition \ref{defn:GCW},
 and for $0\leq k\leq \nu$
 let $U_k\subset X$ be a $G$-invariant neighbourhood of $X_k\setminus X_{k-1}$ inside $X\setminus X_{k-1}$, such that $U_k$ deformation
 retracts $G$-equivariantly onto $X_k\setminus X_{k-1}$. Each difference
 $X_k\setminus X_{k-1}$ takes the form $\coprod_{i\in I_k} G/H_i\times \mathring{D}(\epsilon^k)$,
 where $I_k$ is the finite set indexing cells of index $k$ in $X$, $H_i$ is a closed subgroup of $G$
 for all $i\in I_k$, and $\mathring{D}$ denotes the open unit disc. There is a $G$-equivariant
 retraction of $G/H_i\times \mathring{D}(\epsilon^k)$ onto its core $G/H_i\times \set{0}$;
 operating these retractions simultaneously shows that $U_k$ is $G$-categorical.
\end{proof}
Note that in the previous proof we crucially used our non-standard definition of the $G$-equivariant
Lusternik-Schnirelman category.

\section{Strategy of proof of Theorem \ref{thm:main}}
We fix $m\geq 1$ and consider the complex vector space $\C^m$ with complex coordinates
$z_1,\dots,z_m$; the real coordinates of $\C^m$ are $x_i=\Re(z_i)$ and $y_i=\Im(z_i)$ for all $1\leq i\leq m$.
% We consider the standard Euclidean metric on $\C^m$ and on its open subspaces.
Let $\Delta(z_1,\dots,z_m)$ be a non-zero polynomial with complex coefficients and denote
by $\V\subset\C^m$ the subspace of points of $\C^m$ where $\Delta$ does not vanish.
We write explicitly
% also assume for simplicity that $\Delta$ has no constant term, i.e. $\Delta(0,\dots,0)=0$,
% and we write
\[
 \Delta(z_1,\dots,z_m)=\sum_{\ui\in I}\alpha_{\ui}\cdot z_1^{i_1}\dots z_m^{i_m},
\]
where $I\subset\Z_{\geq0}^m$ is a finite set of multiindices
$\ui=(i_1,\dots,i_m)$, and $\alpha_{\ui}\in\C\setminus\set{0}$ for all $\ui\in I$.

\begin{defn}
\label{defn:homogeneisation}
A \emph{homogeneisation} of $\Delta$ is the assignment of a degree $d_i\in\Z$ to each
of the variables $z_i$, such that $\Delta$ becomes a homogeneous polynomial of some degree $N=N(d_1,\dots,d_m)$.
We represent a homogeneisation by a vector of integers $(d_1,\dots,d_m)\in\Z^m$.
The set of all homogeneisations of $\Delta$ is denoted by $\Homog(\Delta)\subset\Z^m$.
\end{defn}
In formulas, the condition for $(d_1,\dots,d_m)$ to lie in $\Homog(\Delta)$ is that the integer
$N=i_1d_1+\dots+i_md_m$ is the same for all $\ui\in I$.

For example, if $m=3$ and $\Delta=z_1^2-z_2z_3$, then $(1,1,1)$ and $(2,4,0)$ are homogeneisations
of $\Delta$: in the two cases $\Delta$ becomes homogeneous, of degree $2$ and $4$ respectively.
Note that $(0,\dots,0)$ is always a homogeneisation of $\Delta$, making it homogeneous
of degree $0$; in some cases this is the only possible homogeneisation, e.g. for $m=2$ and $\Delta=z_1+z_1^2+z_2^2+z_2^3$.
Note that $\Homog(\Delta)$ is a sub-$\Z$-module of $\Z^m$.

\begin{defn}
 \label{defn:scalaraction}
 For $s\geq 0$ let $\T^s=(\Sone)^s$ be the $s$-dimensional torus, considered as a Lie group (it is the trivial
 group for $s=0$). We denote by $\utheta=(\theta_1,\dots,\theta_s)$ a generic element of $\T^s$,
 where $\theta_i\in\Sone\subset\C$ is a unit complex number. Assume that
 we are given a matrix $\Xi=(\xi_{i,j})_{1\leq i\leq s,1\leq j\leq m}$
 with coefficients in $\Z$, whose rows (of length $m$) lie in $\Homog(\Delta)$.
 Then $\Xi$ induces an action of $\T^s$ on $\V$ \emph{by scalar multiplication} by the formula
 \[
  \utheta\cdot(z_1,\dots,z_m)=(z'_1,\dots,z'_m),
 \]
where for all $1\leq j\leq m$ we set $z'_j=\theta_1^{\xi_{1,j}}\cdot\theta_2^{\xi_{2,j}}\dots\theta_s^{\xi_{s,j}}\cdot z_j$.
\end{defn}
The condition that all rows of $\Xi$ lie in $\Homog(\Delta)$ ensures the equality
\[
 \Delta(z'_1,\dots,z'_m)=\theta_1^{N(\xi_{1,1},\dots,\xi_{1,m})}\cdot\theta_2^{N(\xi_{2,1},\dots,\xi_{2,m})}
 \dots\theta_s^{N(\xi_{s,1},\dots,\xi_{s,m})}\cdot \Delta(z_1,\dots,z_m).
\]
In particular $(z_1,\dots,z_m)$ lies in $\V$ (i.e. $\Delta(z_1,\dots,z_m)\neq 0$) if and only if
$(z'_1,\dots,z'_m)$ lies in $\V$ (i.e. $\Delta(z'_1,\dots,z'_m)\neq 0$), and thus the formula gives
a well-defined action of $\T^s$ on $\V$. The action is not supposed to be free or faithful in any way:
for instance we allow $\Xi$ to be the zero matrix, yielding a trivial action of $\T^s$ on $\V$.
In the rest of the section and in the next section we fix a matrix $\Xi$ as in Definition \ref{defn:scalaraction}. We will
prove the following proposition.
\begin{prop}
 \label{prop:main}
 Assume that all stabilisers of the action of $\T^s$ on $\V$
 have dimension $\leq t$, for some $0\leq t\leq s$. Then $\V\times \V$ is $G$-equivariantly homotopy
 equivalent to a $G$-CW complex $X$ with $\ind(X)\leq 2m-s+t$.
\end{prop}
Using Theorem \ref{thm:GMorsefinal}, in order to prove Proposition
\ref{prop:main} it suffices to exhibit a $\T^s$-invariant Morse function
on the manifold $M=\V\times\V$ with indices $\leq 2m-s+t$: we will do it in the next section.
In the remainder of this section we will see how Theorem \ref{thm:main} follows from Proposition \ref{prop:main}.
\begin{lem}
 \label{lem:VGconnected}
 The $\T^s$-space $\V$ is $G$-connected.
\end{lem}
\begin{proof}
 First, note that $\V$ is path-connected, as every complement of the zero-locus of a (non-zero) polynomial
 in $\C^m$ is (this, incidentally, shows that our attempt to extimate $\TC(\V)$ is not
 completely pointless!).
 Let $H<\T^s$ be a non-trivial closed subgroup. For a point $(z_1,\dots,z_m)\in\V$, the property
 that $(z_1,\dots,z_m)$ belongs to $\V^H$ can be characterised as follows: for all $\utheta\in H$
 and for all $1\leq j\leq m$, either $z_j=0$ or $\theta_1^{\xi_{1,j}}\cdot\dots\theta_s^{\xi_{s,j}}=1\in\Sone$.
 In particular every $\utheta\in H$ imposes some restrictions of the form $z_j=0$ for points
 of $\V^H$.
 
 Reasoning the other way around, let $\set{j_1,\dots,j_l}\subset\set{1,\dots,m}$ be the subset of indices $j$
 for which there exists $\utheta\in H$ with $\theta_1^{\xi_{1,j}}\cdot\dots\theta_s^{\xi_{s,j}}\neq1$:
 then $\V^H$ is the intersection of $\V$ with the sub-vector space of $\C^m$ determined
 by the equations $z_{j_i}=0$ for all $1\leq i\leq l$. This intersection is again the complement
 of the zero locus of a polynomial in a smaller complex vector space, so it is either empty (if the polynomial
 restricts to the zero polynomial on the subspace) or path-connected.
\end{proof}

By Lemma \ref{lem:VGconnected} and Proposition \ref{prop:TCvscat} we
have $\TC_G(\V)\leq \cat_G(M\times M)$,
and since the $G$-equivariant Lusternik-Schnirelman category is a homotopy invariant,
we can write $\TC_G(M)\leq \cat_G(X)$,
where $X$ is given by Proposition \ref{prop:main}.
Since $X$ is a $G$-CW complex of index $\leq 2m-s+t$,
using Proposition \ref{prop:CWLS} we obtain $\cat_G(X)\leq 2m-s+t$.

\section{A good Morse function on \texorpdfstring{$\V\times\V$}{VxV}}
\label{sec:goodMorse}
We need to find a $\T^s$-equivariant Morse function $f\colon \V\times \V\to\R$
with indices $\leq 2m-s+t$, using the notation from Proposition \ref{prop:main}.
We will use essentially the same method used by Andreotti and Frankel \cite{AndreottiFrankel}
in their proof of the Lefschetz hyperplane theorem.

We first define a function $g\colon\V\to\R$ by the equation
\[
 g(z_1,\dots,z_m)=|z_1|^2+\dots+|z_m|^2+\left|\frac{1}{\Delta(z_1,\dots,z_m)}\right|^2.
\]
Note that $g$ is invariant with respect to the action of $\T^s$ by scalar multiplication.
Note also that $g$ takes values $\geq0$
and is proper: for $\lambda>0$ let $\set{g\leq\lambda}\subset\V$ denote the subspace
of points $(z_1,\dots,z_m)\in\V$ for which $g(z_1,\dots,z_m)\leq \lambda$;
then $\set{g\leq\lambda}$ is contained in the closed subspace $\set{|\Delta|\geq \frac 1{\sqrt{\lambda}}}\subset\C^m$,
so $\set{g\leq\lambda}$ is also closed in $\C^m$; moreover $\set{g\leq\lambda}$ is contained
in the closed ball of radius $\sqrt{\lambda}$ centred at 0, hence $\set{g\leq\lambda}$ is compact.

The critical locus $\crit(g)\subseteq\V$, i.e. the subspace of points of $\V$ where $dg$ vanishes,
can be characterised by algebraic equations and inequalities in the real coordinates $x_i=\Re(z_i)$ and $y_i=\Im(z_i)$,
hence $\crit(g)$ is a real semi-algebraic subset of $\C^m=\R^{2m}$. As such, $\crit(g)$ has finitely
many connected components (see for instance \cite[Theorem 2.4.4]{RAG}).

Note also that the vanishing of $dg$ on a connected component of $\crit(g)$ implies that $g$ is constant
on this connected component. Since $g$ is proper on $\V$ and since $\crit(g)$ is closed in $\V$, every
component of $\crit(g)$ is compact; since $\crit(g)$ has finitely many connected components, also $\crit(g)$
is compact.

Recall that if $p\in\C^m$ is a point, $U$ is a neighbourhood of $p$ and $h\colon U\to\R$ is a smooth function,
we can define a Euclidean Hessian $\fH(h)_p$ as the $2m\times 2m$ matrix of second partial derivatives,
computed with respect to the Euclidean coordinates $x_i$ and $y_i$ of $\C^m$. The Euclidean Hessian $\fH(h)_p$
can be considered as a bilinear form on $T_p(\C^m)$, but its definition strongly depends on the choice of
the coordinates we use to parametrise a neighbourhood of $p$, in our case the Euclidean coordinates.
On the other hand, if $p$ happens to be a critical
point for $h$, i.e. $dh$ vanishes at $p$, then the Hessian $H(h)_p$ is a bilinear form on $T_p\C^m$ which
only depends on $p$ and $g$ but not on the coordinates used to parametrise a neighbourhood of $p$;
moreover $\fH(h)_p$ and $H(h)_p$ agree if $p\in\crit(h)$.

\begin{lem}
 \label{lem:symmetricsignature}
 Let $p\in\C^m$ be a point and let $U$ be an open neighbourhood of $p$;
 let $\eta\colon U\to\C\setminus\set{0}$ be a holomorphic function, and define $h\colon U\to \C^m$ by $h(p)=|\eta(p)|^2$.
 Consider the Euclidean Hessian $\fH(h)_p$: then the negativity index of $\fH(h)_p$ is at most $m$.
\end{lem}
\begin{proof}
Up to a change of coordinates of $U$ by translation in $\C^m$, we can assume $p=0$.
We consider the Taylor expansion of $\eta$ around 0:
\[
 \eta(z_1,\dots,z_m)=c+\sum_{i=1}^m\lambda_i z_i+\sum_{i,j=1}^m \mu_{i,j}z_iz_j+\mathrm{l.o.t.},
\]
where $c=\eta(0)$, $\lambda_i=\frac{\partial \eta}{\partial z_i}(0)$, $\mu_{i,j}=\frac{\partial^2 \eta}{\partial z_i\partial z_j}$
and the lower order terms are those of degree $\geq 3$. Up to multiplying $\eta$ by $1/c$ (and consequently
$h$ by the positive number $1/|c|^2$), we can assume $c=1$.
Up to a $\C$-linear change of coordinates of $U$, we can further assume
that the $m\times m$ matrix $(\mu_{i,j})$ is diagonal, with entries $\mu_{1,1},\dots, \mu_{r,r}$ equal to 1 and entries
$\mu_{r+1,r+1},\dots,\mu_{m,m}$ equal to 0, for some $0\leq r\leq m$. We can therefore assume that the Taylor
expansion of $\eta$ has the form:
\[
 \eta(z_1,\dots,z_m)=1+\sum_{i=1}^m\lambda_i z_i+z_1^2+z_2^2+\dots+z_r^2+\mathrm{l.o.t.}.
\]
% Let $\Lambda$ be the $m\times 1$ matrix with coefficients $\lambda_1,\dots,\lambda_m$
% Then the Taylor expansion of $h$ near 0 is the following,
% where we write $\lambda_i=a_i+\sqrt{-1}b_i$, and $M$ is the $2m\times 2$ matrix with $M_{2i+1,1}=M_{2i+2,2}=a_i$
% and $M_{2i+1,2}=-M{2i+2,1}=b_i$ for all $0\leq i\leq m-1$:

Then the Taylor expansion of $h$ is the following, using the variables $z_1,\dots,z_m$ and their conjugates $\bar z_1,\dots,\bar z_m$,
and recalling the identity $h=\eta\cdot\bar\eta$:
\[
 h(z_1,\dots,z_m)=1+\pa{\sum_{i=1}^m \lambda_iz_i+\bar\lambda_i\bar z_i} + 
 \pa{\sum_{i,j=1}^m \lambda_i\bar\lambda_j z_i\bar z_j} + \pa{\sum_{i=1}^r z_i^2+\bar z_i^2}+\mathrm{l.o.t.}
\]
% \[
%  h(x_1,\dots,x_m,y_1,\dots y_m)=1+\pa{\sum_{i=1}^m a_ix_i-b_iy_i} + 
%  \big(\sum_{i,j=1}^m \lambda_i\bar\lambda_j z_i\bar z_j\big) + \big(\sum_{i=1}^r z_i^2+\sum_{i=1}^r\bar z_i^2\big).
% \]
The Euclidean Hessian $\fH(h)_0$ can be regarded as the sum of two $\R$-bilinear forms $\fH^1$ and $\fH^2$,
the first corresponding
to the quadratic form $\big(\sum_{i,j=1}^m \lambda_i\bar\lambda_j z_i\bar z_j\big)$,
% (which can be
% written as a quadratic form in the variables $x_1,y_i$ with coefficients expressed in terms of the numbers $a_i,b_i$),
the second corresponding to the quadratic form  $\big(\sum_{i=1}^r z_i^2+\bar z_i^2\big)$.

The bilinear form $\fH^2$ can be written, with respect to the coordinates $x_i$ and $y_i$,
as a $2m\times 2m$ real diagonal matrix,
with $r$ occurrences of $2$, $r$ occurrences of $-2$ and $2m-2r$ occurrences of $0$ on the main diagonal:
this follows from the equality $z_i^2+\bar z_i^2=2x_i^2-2y_i^2$ for all $1\leq i\leq r$. In particular
$\fH^2$ has signature $(2m-2r,r,r)$, so it has negativity index $r\leq m$.

The bilinear form $\fH^1$ comes from the sesquilinear form on $\C^m$ represented by the $m\times m$ complex matrix
$\Lambda\cdot\bar\Lambda^{T}$, where $\Lambda$ is the $m\times 1$ matrix with entries $\lambda_1,\dots,\lambda_m$.
Every sesquilinear form of type $\Lambda\cdot\bar\Lambda^{T}$ is semidefinite positive on $\C^m$, and hence semidefinite
positive when considered as a $\R$-bilinear form on $\R^{2m}$; hence $\fH^1$ is semidefinite positive.

It follows that the negativity index of $\fH(h)_0=\fH^1+\fH^2$ is at most $r\leq m$.
\end{proof}
We apply Lemma \ref{lem:symmetricsignature} to the function $\eta(z_1,\dots,z_m)=1/\Delta(z_1,\dots,z_m)$,
obtaining that for all $p\in \V$ the matrix $\fH(|\eta|^2)$ has negativity index $\leq m$;
in other words, the sum of its nullity and positivity indices is at least $m$.

On the other hand $\fH(|z_1|^2+\dots+|z_m|^2)_p$ is (strictly) positive definite:
we have indeed that $\fH(|z_1|^2+\dots+|z_m|^2)_p$ is the $2m\times 2m$ identity matrix, independently of $p\in\V$.
We obtain therefore that for all $p\in\V$ the matrix $\fH(g)_p$ has positivity index $\geq m$.

Now we consider the function $\tilde f\colon \V\times \V\to \R$ defined by $\tilde f(p,p')=g(p)+g(p')$.
Note that $\tilde f$ is invariant with respect to the diagonal action of $\T^s$,
takes values $\geq 0$ and is proper.

The differential of $\tilde f$ is the ``direct'' sum
of the pullbacks of the differentials of $g$ at $p$ and $p'$: this implies that a critical point $(p,p')$
for $\tilde f$ is precisely a couple of critical points $p,p'$ for $g$; as a consequence
$\crit\pa{\tilde f}=\crit(g)\times\crit(g)$ is compact.

The Euclidean Hessian $\fH(\tilde f)_{(p,p')}$ computed at any point $(p,p')\in\V\times \V$ is the block
sum of the Euclidean Hessians $\fH(g)_p$ and $\fH(g)_{p'}$: in particular the positivity
index of $\fH(\tilde f)_{(p,p')}$ is at least $2m$ for all $(p,p')\in\V\times \V$.

By Lemma \ref{lem:density} we can find a $\T^s$-invariant Morse approximation $f\colon\V\times\V\to\R$ of $\tilde f$;
we can ensure that $\tilde f$ and $f$ agree outside a $\T^s$-invariant relatively compact open subspace $U\subset\V\times \V$
containing $\crit\pa{\tilde f}$ in its interior; and we can ask that $\tilde f$ and $f$ are close in the $C^2$ norm,
in particular that also $f$, as $\tilde f$, has the property that $\fH(f)_{(p,p')}$ has positivity index at least
$2m$ for all $(p,p')\in U$.

We claim that $f$ has critical points of index $\leq 2m+s-t$. Let $(p,p')$ be a critical point for $f$:
then $H(f)_{(p,p')}$ can be represented in coordinates as $\fH(f)_{(p,p')}$, and thus it has positivity
index at least $2m$.
Denote by $N\subset \V\times \V$ the critical manifold containing $(p,p')$;
then $N\cong \T^s/H$, for some closed subgroup $H<\T^s$. We have that $H$ is the intersection
of the stabilisers of $p$ and $p'$ by the action of $\T^s$ on $\V$, therefore our hypothesis
on $t$ implies that $\dim H\leq t$, and hence $\dim N\geq s-t$.
Since the bilinear form $H(f)_{(p,p')}$ vanishes on $T_{(p,p')}N$,
we have that the nullity index of $H(f)_{(p,p')}$ is at least $s-t$. It follows that the negativity
index of $H(f)_{(p,p')}$ is at most $4m-(2m+s-t)=2m-s+t$.

\section{Applications}
In this section we apply Theorem \ref{thm:main} to configuration spaces.
Fix $n\geq 2$, and note that $F_n$ deformation
retracts onto its subspace $F^0_n$ of ordered configurations $(w_1,\dots,w_n)$
whose barycentre $(w_1+\dots+w_n)/n$ is equal to $0$. The explicit
deformation retraction $\rho\colon F_n\times[0,1]\to F_n$ is given by the following formula:
\[
 \rho(w_1,\dots,w_n;t)=\pa{w_1-t\frac{w_1+\dots+w_n}n,\dots,w_n-t\frac{w_1+\dots+w_n}n}.
\]
Note that $F^0_n$ is a $\T$-invariant subspace of $F_n$ and that
$\rho$ is a $\T$-equivariant deformation retraction, so that $F^0_n$ and $F_n$ are $\T$-equivariantly
homotopy equivalent: we can study the space $F^0_n$ instead of $F_n$.
Note also that $\rho$ is equivariant with respect to the action of the symmetric
group $\fS_n$ on $F_n$ by permutation of labels: thus $\rho$ induces a $\T$-equivariant deformation
retraction of $C_n$ onto its $\T$-invariant subspace $C^0_n$ of unordered configurations $\set{w_1,\dots,w_n}$
with barycentre $(w_1+\dots+w_n)/n$ equal to $0$.

% We can use the coordinates $z_1,\dots,z_{n-1}$ to parametrise $F^0_n$ as a subspace of $\C^{n-1}$,
% since the last coordinate $z_n$ is equal to $-z_1-\dots-z_{n-1}$ on $F^0_n$. In this way
% $F^0_n$ is identified with the subspace $\V\subset\C^{m-1}$ of points $(z_1,\dots,z_{n-1})$
% for which the polynomial
% \[
%  \Delta(z_1,\dots,z_{n-1}):=\pa{\prod_{1\leq i<j\leq n}
% \]

We first focus on the unordered case, so we want to apply Theorem \ref{thm:main} to the space $C_n^0$.
For each $\set{w_1,\dots,w_n}\in C_n^0$ we can form the polynomial
\[
 P(w)=w^n+\sum_{i=2}^n(-1)^ia_iw^{n-i}:=(w-w_1)\cdot(w-w_2)\cdot\dots\cdot(w-w_n),
\]
where the coefficient
$a_i$ is the $i$-th elementary symmetric function in $w_1,\dots,w_n$;
note that $a_1=\sum_{i=1}^nw_i$ vanishes by hypothesis, therefore we have omitted the corresponding term from the sum above.
The discriminant
\[
\Delta(P)=\prod_{1\leq i<j\leq n}(w_i-w_j)^2
\]
is a symmetric polynomial in the variables $w_1,\dots,w_n$, hence
it can be expressed as a polynomial $\Delta^C(a_2,\dots,a_n)$ in the elementary symmetric functions
$a_2,\dots,a_n$: again we omit $a_1$ because this elementary
symmetric function is constantly zero on $C_n^0$.

Note also that for $2\leq i\leq n$ the elementary symmetric function $a_i$ has
total degree $i$ in the variables $w_i$, whereas the polynomial $\Delta^C$ has
total degree $n(n-1)$: in particular, for all $\theta\in\T=\Sone$,
if we rotate the configuration $\set{w_1,\dots,w_n}\in C^0_n$ by $\theta$, i.e. we replace it with
the configuration $\set{\theta w_1,\dots,\theta w_n}$, then the symmetric function
$a_i$ is multiplied by $\theta^i$ and the value of $\Delta^C$ is multiplied by $\theta^{n(n-1)}$.

Consider now $\C^{n-1}$ with coordinates $z_2,\dots,z_n$, and let $\V^C$ be the complement
of the zero locus of the polynomial $\Delta^C(z_2,\dots,z_n)$.
The fundamental theorem of algebra establishes a homeomorphism between
$C^0_n$ and $\V^C$; moreover the previous discussion shows that $(2,3,\dots,n)$ is a homogeneisation
of $\Delta^C(z_2,\dots,z_n)$, so we can consider the corresponding action of $\T$ on $\V^C$ by scalar
multiplication: this action corresponds under the homeomorphism $C^0_n\cong \V^C$ to the action
of $\T$ on $C^0_n$ by rotations.

It is left to check that the stabilisers of the action of $\T$ on $\V^C$ have dimension at most 0,
i.e. they are finite groups. Since the only closed subgroups of $\T$ are either finite cyclic groups
or the entire $\T$, it suffices to check that no point of $\V^C$ is fixed by the entire $\T$,
or equivalently no point of $C^0_n$ is fixed by the entire $\T$: the second statement is a straightforward
consequence of the assumption $n\geq 2$.

Applying Theorem \ref{thm:main} we obtain that $\TC_{\T}(C^0_n)=\TC_{\T}(\V^C)\leq 2(n-1)-1+0=2n-3$.

We pass now to the ordered case. Similarly as above,
we can parametrise $F^0_n$ using the coordinates $w_1,\dots,w_{n-1}$, because on $F^0_n$ we can express
the last coordinate $w_n=-(w_1+\dots+w_{n-1})$. The defining inequality for $F_n\subset\C^m$,
which is $\prod_{1\leq i<j\leq n}(w_i-w_j)\neq 0$, restricts to the following defining inequality for $F^0_n$, expressed
only in terms of the variables
$w_1,\dots,w_{n-1}$:
\[
 \Delta^F(w_1,\dots,w_{n-1}):=\!\pa{\prod_{1\leq i<j\leq n-1}(w_i-w_j)\!}\pa{\prod_{i=1}^{n-1}(w_i+(w_1+\dots+w_{n-1}))}\neq 0.
\]
In particular $F^0_n$ is homeomorphic to the complement $\V^F\subset\C^{n-1}$ of the zero locus of the polynomial
$\Delta^F(z_1,\dots,z_{n-1})$. Since $\Delta^F$ is homogeneous in the variables $z_1,\dots,z_{n-1}$
we have that $(1,\dots,1)$ is a homogeneisation for $\Delta^F$; the corresponding action of $\T$ on $\V$
by scalar multiplication corresponds, under the homeomorphism $F^0_n\cong\V^F$, to the action of $\T$
on $F^0_n$ by rotation.

It is left to check that the stabilisers of the action of $\T$ on $\V^F$, or equivalently on $F^0_n$,
are finite groups, and this follows again, as in the unordered case, from the inequality $n\geq 2$. Theorem \ref{thm:main}
applies again to give $\TC_{\T}(F^0_n)=\TC_{\T}(\V^F)\leq 2(n-1)-1+0=2n-3$.

\section{Outlook: finding explicit motion planners}
In principle one can use the strategy of the proof of Theorem \ref{thm:upperbound} to construct an explicit
$\T$-equivariant motion planner on $C^n$ (or on $F^n$). One fixes a $\T$-invariant
Morse function $f\colon C^0_n\times C^0_n\to\R$, and makes a list of its critical manifolds
$N_1,\dots,N_r\subset C^0_n\times C^0_n$. For each $N_i\cong\T/H_i$, where $H_i$ is some finite cyclic group,
one chooses a point $(q_i,q'_i)\in N_i$ and fixes a path from $q_i$ to $q'_i$ inside the subspace $(C_n^0)^{H_i}$,
which we know is path-connected (see Lemma \ref{lem:VGconnected}); we extend this path, in the unique possible way,
to a $\T$-equivariant motion rule on each $N_i$.

Now, given two configurations $p,p'\in C_n$, the motion planner will connect $p$ to $p'$ as follows:
\begin{enumerate}
 \item we connect
both $p$ and $p'$, respectively, with the configurations $\rho(p,1)$ and $\rho(p',1)$ in $C^0_n$,
using the homotopy $\rho$;
\item we use the gradient flow of $f$ on $C^0_n\times C^0_n$ to move the couple $(\rho(p,1),\rho(p',1))$
to a new couple $(q,q')$ belonging to some critical manifold $N_i$, thus connecting $p$ to $q$
and $p'$ to $q'$;
\item we connect $q$ to $q'$ using the fixed motion rule on $N_i$.
\end{enumerate}
We can define, for all $0\leq i\leq 2n-4$, the subspace $K_i\subset C_n\times C_n$ as the (not necessarily
open) subspace containing pairs $(p,p')$ that in step (2) are mapped to a pair $(q,q')$ belonging
to a critical manifold $N_j$ of index precisely $i$. Then the hope is that the above motion rule is continuous
on each $K_i$: if this is the case, one can argue that each $K_i$ admits a $\T$-invariant open neighbourhood $U_i$
that deformation retracts onto $K_i$, and thus one 
can define a continuous motion rule on $U_i$. Since $K_i$ is a subspace of a manifold, one can expect
that the deformation of $U_i$ onto $K_i$ can be given by a very explicit formula, for example by
selecting, for all $(p,p')\in U_i$, the shortest path (with respect to the Euclidean metric on $\V^C\subset\C^{m-1}$)
joining $(p,p')$ with a pair in $K_i$; so one can expect that passing from $K_i$ to $U_i$ is not a big issue.
As far as I can see, there are three other, major issues in implementing this strategy.
\begin{itemize}
\item Lemma \ref{lem:density} does not give an explicit
way to perturb the explicit function $\tilde f$ given in Section \ref{sec:goodMorse}
to a $\T$-invariant Morse function $f$.
\item The motion rule described above may not be continuous
on the subspaces $K_i$: this problem is reflected in the fact that Theorem \ref{thm:Wasserman}
only produces a $\T$-cell complex, and we need to further subdivide cells (and to change the attaching
maps by a $\T$-equivariant homotopy) in order to obtain a $\T$-CW complex.
\item We need to list all critical manifolds $N_i$ for $f$.
\end{itemize}
For the first issue, one can use directly the gradient flow of $\tilde f$ to perturb a pair
in $C^0_n\times C^0_n$ to one in $\crit\pa{\tilde f}=\crit(g)\times \crit(g)$, using the function
$g\colon C^0_n\to\R$. Even assuming that $g$ is Morse (something we did not speculate on, and for a good reason, as
we will see soon), the connected components of $\crit\pa{\tilde f}$ have the form $\T/H_1\times \T/H_2$ and are thus 2-dimensional,
so $\tilde f$ is \emph{surely} not a $\T$-invariant Morse function.
Nevertheless one can fairly easily decompose each component
$\T/H_1\times \T/H_2$ of $\crit\pa{\tilde f}$ into 2 $\T$-invariant pieces, each of which has
a $\T$-equivariant deformation retraction onto an orbit of the action of $\T$ on $C^0_n\times C^0_n$.
So up to assuming that $g$ is Morse, the first issue can be solved or bypassed.

For the second issue, we need to understand the representation theory (over $\R$)
of $\T$ and of its finite subgroups, in order to decompose the $\T$-cell complex from Theorem \ref{thm:Wasserman}
into a $\T$-pre-CW complex: this is due to the actual construction underlying Theorem \ref{thm:decomposition}.
The representation theory of $\T$ and of finite cyclic groups is well understood,
so we can assume to be able to obtain an explicit $\T$-pre-CW complex $X$ onto which $C^0_n\times C^0_n$
deformation retracts $\T$-equivariantly. We can also hope that $X$ is already
a $\T$-CW complex, and redefine $K_i$ as the set of pairs $(p,p')$ that are mapped by the first two
steps of the motion planner to a pair $(q,q')$ in $X_i\setminus X_{i-1}$.
Let us assume for a moment that the second issue can also solved.

The third issue is, in my opinion, the biggest obstacle to implement an explicit
motion planner. Understanding
$\crit\pa{\tilde f}=\crit(g)\times\crit(g)$ boils down to understanding $\crit(g)$,
and even assuming that $g$ is a $\T$-invariant Morse function
we would need to make a list of the components of $\crit(g)$. And anyway,
we would still have to prove that $g$ is Morse!

We have given an explicit formula for $g$, and all properties that we used are essentially
the following: $g$ is a proper $\T$-invariant function on $C^0_n$ obtained as a sum of
absolute values of (locally defined) algebraic holomorphic functions.

Let us try to replace $g\colon C^0_n\to\R$ with another, similar function $g'\colon C^0_n\to \R$,
given by the formula
\[
 g'(w_1,\dots,w_n)=\sum_{i=1}^n|w_i|^2+\sum_{1\leq i<j\leq n}\frac{1}{|w_i-w_j|},
\]
for all $\set{w_1,\dots,w_n}\in C^0_n$.

The configurations in $\crit(g')$ are known in celestial mechanics
as the \emph{central configurations for the planar $n$-body problem}; for $n\leq 7$
it is known that $g'$ is a $\T$-invariant Morse function, and there is a complete list of components
of $\crit(g')$ \cite{CC}, but for higher $n$ there is no such a list, and it is even
unknown whether $g$ is a $\T$-equivariant Morse function or not: this latter question
is also known as the 6th problem in Smale's eighteen problems for the 21st century.
If the third issue is such a hard problem for $g'$, I do not expect that
it can be easily solved for $g$ (or for any other function with similar properties).

\bibliography{Bibliography.bib}{}
\bibliographystyle{plain}

\end{document}

%% file: macros.tex
\newcommand{\C}{\mathbb{C}} %complex plane
\newcommand{\fH}{\mathfrak{H}} %Euclidean Hessian
\newcommand{\R}{\mathbb{R}} %real numbers
\newcommand{\fS}{\mathfrak{S}} %symmetric group
\newcommand{\T}{\mathbb{T}}
\newcommand{\V}{\mathcal{V}} %a variety
\newcommand{\Z}{\mathbb{Z}} %integers
\newcommand{\TC}{\mathbf{TC}}
\newcommand{\cat}{\mathbf{cat}}
\newcommand{\utheta}{\underline{\theta}}
\newcommand{\ind}{\mathrm{ind}} %index of a critical manifold
\newcommand{\ui}{\underline{i}} %multiindex
\newcommand{\Homog}{\mathrm{Homog}} %set of all homogeneisations
\newcommand{\crit}{\mathrm{crit}} %critical locus of a function

\newcommand{\pa}[1]{\left(#1\right)}

\newcommand{\set}[1]{\left\{#1\right\}}

\newcommand{\Sone}{\mathbb{S}^1}

\renewcommand{\phi}{\varphi}
\renewcommand{\epsilon}{\varepsilon}

\DeclareMathOperator{\Id}{Id}

%% file: tcvarieties.bbl
\begin{thebibliography}{10}
\bibitem{AndreottiFrankel}
A.~Andreotti and T.~Frankel.
\newblock The lefschetz theorem on hyperplane sections.
\newblock {\em Annals of Mathematics}, 69(3):713--717, 1959.

\bibitem{Bianchi:HcPn}
Andrea Bianchi.
\newblock On the homology of the commutator subgroup of the pure braid group.
\newblock \href{https://arxiv.org/abs/1905.05099}{arXiv:1905.05099}, 2019.

\bibitem{BianchiRecio}
Andrea Bianchi and David Recio-Mitter.
\newblock Topological complexity of unordered configuration spaces of surfaces.
\newblock {\em Algebraic and Geometric Topology}, 19:1359--1384, 2019.

\bibitem{RAG}
J.~Bochnak, M.~Coste, and M.-F Roy.
\newblock {\em Real Algebraic Geometry}, volume~36 of {\em A Series of Modern
  Surveys in Mathematics}.
\newblock Springer, 1996.

\bibitem{CG12}
H.~Colman and M.~Grant.
\newblock Equivariant topological complexity.
\newblock {\em Algebraic and Geometric Topology}, 12(4):2299--2316, 2012.

\bibitem{Farber}
Michael Farber.
\newblock Topological complexity of motion planning.
\newblock {\em Discrete and Computational Geometry}, 29:211--221, 2003.

\bibitem{FY}
Michael Farber and Sergey Yuzvinsky.
\newblock Topological robotics: Subspace arrangements and collision free motion
  planning.
\newblock {\em Transl. of AMS}, 212:145--156, 2004.

\bibitem{Illman}
S\"oren Illman.
\newblock The equivariant triangulation theorem for actions of compact lie
  groups.
\newblock {\em Math. Ann.}, 262:487--501, 1983.

\bibitem{Marzantowicz}
W~Marzantowicz.
\newblock A g-lusternik–schnirelman category of space with an action of a
  compact {L}ie group.
\newblock {\em Topology}, 28:403--412, 1989.

\bibitem{CC}
M.~Moczurad and P.~Zgliczy\'{n}ski.
\newblock Central configurations in planar n-body problem with equal masses for
  n=5,6,7.
\newblock {\em Celestial Mechanics and Dynamical Astronomy}, 131(46), 2019.

\bibitem{tomDieck}
Tamo {tom Dieck}.
\newblock {\em Transformation Groups}, volume~8 of {\em Studies in
  Mathematics}.
\newblock De Gruyter, 1987.

\bibitem{Wasserman}
A.~Wasserman.
\newblock Equivariant differential topology.
\newblock {\em Topology}, 8:127--150, 1969.

\end{thebibliography}
